%% file: main.tex
\documentclass[a4paper,11pt]{article}
\usepackage[utf8]{inputenc}
\usepackage[T1]{fontenc}
\usepackage[english]{babel}
\usepackage{amsmath,amssymb,amsthm}
\usepackage[short]{optidef}				% Package for linear programming
\usepackage[shortlabels]{enumitem} 
\usepackage{comment}
\usepackage{standalone}
\usepackage{mathtools}
\usepackage{fullpage}
\usepackage[dvipsnames]{xcolor}
\usepackage{float}
\usepackage{mathdots}
\usepackage{tikz}
\usepackage[subrefformat=parens,labelformat=parens]{subcaption}
\graphicspath{{./Figuras/}} %%%% path for figures
\usetikzlibrary{graphs,graphs.standard}
\usepackage{hyperref}
\usepackage{thmtools}
\usepackage{thm-restate}
\tikzstyle{line} = [draw, thick]

\usepackage{graphicx}
\newtheorem*{citedtheorem}{Theorem}

\definecolor{myblue}{RGB}{80,80,160}
\definecolor{mygreen}{RGB}{80,160,80}
\definecolor{myred}{RGB}{255,0,0}
\definecolor{mybrown}{RGB}{139,69,19}

\newtheorem{theorem}             {Theorem}
\newtheorem{lemma}     	[theorem] {Lemma}

\newtheorem{corollary}	[theorem] {Corollary}
     
\newtheorem{claim}{Claim}

\newtheoremstyle{case}{}{}{}{}{\bfseries}{:}{ }{}
\theoremstyle{case}

\numberwithin{subcase}{case}

\begin{document}

\title{Distance-$k$ Domination Number in Triangular Matchstick Graphs}

\author{
Juan Gutiérrez\thanks{This work was partially supported by Fondo Semilla UTEC 2025.}
\and
Jorge Neira\footnotemark[1]
}
\date{}
\maketitle

\begin{center}
\footnotesize
Department of Computer Science\\
University of Engineering and Technology (UTEC), Lima, Peru\\
E-mail: \texttt{\{jgutierreza,jorge.neira\}@utec.edu.pe}
\end{center}

\input{sections/abstract}
\input{sections/intro}

\input{sections/1_correction/reason}

\input{sections/2_generalization/generalization}

\input{sections/3_improvement/improvement}

\input{sections/4_radius/radius}

\newpage

\bibliographystyle{plain}
\bibliography{referencias}

\end{document}

%% file: sections/abstract.tex
\begin{abstract}
A distance-$k$ dominating set of a graph is a set of vertices such that every vertex lies within distance $k$ of some vertex in the set; its minimum size is the distance-$k$ domination number $\gamma_k$. 
We study $\gamma_k$ for triangular matchstick graphs $T_d$.
Harris et al. (2020) claimed two upper bounds for the cases $k=1$ and $k=2$.
These bounds are shown to be incorrect and are replaced by a corrected general upper bound for $\gamma_k(T_d)$ obtained by refining their tiling method.
For $k=1$ and $d=7q+\beta$, where $0\leq \beta\leq 6$, we further sharpen this construction and prove
$\gamma(T_{7q+\beta})\leq 3.5q^2+(\beta+4.5)q+(\beta-1)$.
Finally, we determine the radius of $T_d$ as $r_d=\left\lceil 2d/3\right\rceil$, which implies $\gamma_k(T_d)=1$ whenever $k\geq r_d$.
\end{abstract}


%% file: sections/intro.tex
\section{Introduction and Preliminaries}

We consider simple connected graphs and use standard terminology \cite{bondy_graph_2008, diestel_graph_2017}. 
For vertices $u$ and $v$, let $dist(u,v)$ denote their distance. A set $D \subseteq V$ is a \emph{distance-$k$ dominating set} if every vertex is at distance at most $k$ from some vertex of $D$. The minimum size of such a set is the \emph{distance-$k$ domination number} $\gamma_k(G)$.
For $k=1$, $\gamma_1(G)$ is the classical domination number $\gamma(G)$, a parameter studied extensively in the literature \cite{haynes_topics_2020, slater_r_1976}.

Domination and distance-$k$ domination have been extensively studied in planar graphs. Matheson and Tarjan conjectured that sufficiently large triangulated planar graphs satisfy $\gamma(G)\le n/4$. In particular, this bound holds for triangular matchstick graphs \cite{hongo_dominating_2013}.
Several additional results concerning this conjecture have been established 
\cite{christiansen_2024, king_pelsmajer_2010, matheson_dominating_1996, plummer_zha_2016, spacapan_domination_2020}.

The \textit{triangular lattice} $T_\infty$ is the infinite graph whose vertex set is
\[
(x,y)=\left(a-\frac{b}{2},\;\frac{b\sqrt{3}}{2}\right), \quad a,b\in\mathbb{Z},
\]
where two vertices are adjacent if their distance is $1$.
A \textit{triangular matchstick graph} $T_d$ is the subgraph of $T_\infty$ induced by vertices with 
$0\le b\le a\le d$. Figure~\ref{fig:mt:graph_tri} illustrates $T_d$ for $0\le d\le3$.
Harris et al. stated upper bounds for the distance-$k$ domination number of triangular matchstick graphs, for $k \in \{1,2\}$, using a periodic pattern on the triangular lattice. Their Theorems~4.2 and~4.3 state the following bounds.
\begin{figure}[ht]
  \centering

  \begin{subfigure}[b]{0.18\textwidth}
    \centering
    \includegraphics{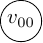}
    \caption{$T_0$.}
  \end{subfigure}
  \hfill
  \begin{subfigure}[b]{0.18\textwidth}
    \centering
    \includegraphics{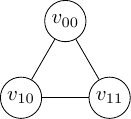}
    \caption{$T_1$.}
  \end{subfigure}
  \hfill
  \begin{subfigure}[b]{0.25\textwidth}
    \centering
    \includegraphics{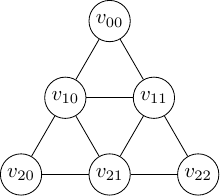}
    \caption{$T_2$.}
  \end{subfigure}
  \hfill
  \begin{subfigure}[b]{0.35\textwidth}
    \centering
    \includegraphics{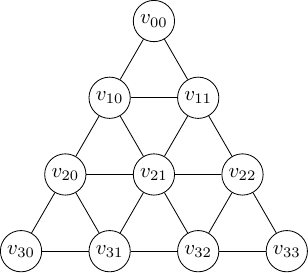}
    \caption{$T_3$.}
  \end{subfigure}

  \caption{$T_d$ for $0 \leq d \leq 3$.}
  \label{fig:mt:graph_tri}
\end{figure}

\begin{citedtheorem}[\cite{harris_2020}, Theorem~4.2]
If $d = 7q + \beta$, with $0 \le \beta \le 6$, then
\[
\gamma(T_d) \le
\begin{cases}
3q(q+2), & \text{if } \beta = 0,\\
3(q+1)(q+3), & \text{if } \beta \neq 0.
\end{cases}
\]
\end{citedtheorem}

\begin{citedtheorem}[\cite{harris_2020}, Theorem~4.3]
If $d = 19q + \beta$, with $0 \le \beta \le 18$, then
\[
\gamma_2(T_d) \le
\begin{cases}
9q(q+1), & \text{if } \beta = 0,\\
9(q+1)(q+2), & \text{if } \beta \neq 0.
\end{cases}
\]
\end{citedtheorem}

However, these upper bounds are shown here not to hold. 
The proofs of these statements overlook vertices of degree~$6$ at intersections of the triangular copies used in the tiling, leading to incorrect double-counting adjustments.
As a consequence, the resulting upper bounds do not hold for $T_d$.
This is further evidenced by the lower bound 
given in the same paper \cite[Lemma~3.1]{harris_2020}, which asymptotically exceeds those upper bounds for large $d$. 

In this work, we study $\gamma_k$ for triangular matchstick graphs. 
We begin by discussing the tiling-based upper-bound argument for triangular matchstick graphs in Section~\ref{sec:harris-discussion}. After that, in Section~\ref{sec:generalization}, 
we generalize the corrected construction to distance-$k$ domination for every $k\geq 1$.
    
\begin{restatable}{theorem}{general}\label{thm:general-upper-bound}
Let $k\geq 1$ and let $D_k=3k^2+3k+1$. If $d=D_kq+\beta$, with $q\geq 1$ and $0\leq \beta \leq D_k-1$, then
\[
\gamma_k(T_d)
\leq
\begin{cases}
\dfrac{D_kq^2+(6k+3)q+2}{2}, & \text{if } \beta=0,\\[6pt]
\dfrac{D_k(q+1)^2+(6k+3)(q+1)+2}{2}, & \text{if } \beta>0.
\end{cases}
\]
\end{restatable}
       
For the case $k=1$, we present an improved upper bound in Section~\ref{sec:improvement}.

\begin{restatable}{theorem}{improvementkone}\label{teorema:upperbound_k1}
If $d = 7q + \beta$, where $q \geq 1$ and $0 \leq \beta \leq 6$, then
\[
\gamma(T_d) \leq 3.5q^2 + (\beta + 4.5)q + \beta - 1.
\]
\end{restatable}

Finally, we determine the radius of triangular matchstick graphs and relate it to $\gamma_k$ in Section~\ref{sec:radius}, which allows us to precisely characterize when $\gamma_k(T_d)=1$.

\begin{restatable}{theorem}{radius}\label{teo:radio}
For every $d \ge 0$, the radius of $T_d$ is
$
\left\lceil \frac{2d}{3} \right\rceil.
$
\end{restatable}

%% file: sections/1_correction/reason.tex
\section{Discussion of the argument of Harris et al.}\label{sec:harris-discussion}
In this section, we demonstrate that the arguments in the proofs of \cite[Theorems~4.2 and~4.3]{harris_2020} are incorrect and identify the double-counting issue that must be corrected.

We denote by $\ell_x$, $\ell_y$, and $\ell_z$ the three sides of the outer triangular boundary of $T_d$, as illustrated in Figure~\ref{fig:mm:correccion1}. Specifically, $\ell_x$ denotes the horizontal base of $T_d$, while $\ell_y$ and~$\ell_z$ denote the right and left boundary sides, respectively.

Harris et al. based their constructions on a periodic pattern $P_{t,r}$ on the infinite triangular grid $T_\infty$~\cite[Theorem~4.1]{harris_2020}. 
For our purposes, this pattern is the periodic set $P_{k+1,1}$ obtained by placing vertices at all points of the form
\[
\bigl((2k+1)x+ky\bigr)\alpha_1+\bigl(kx+(2k+1)y\bigr)\alpha_2,\qquad x,y\in\mathbb{Z},
\]
where $\alpha_1=(1,0)$ and $\alpha_2=\left(-\frac{1}{2},\frac{\sqrt{3}}{2}\right)$ are the two basis directions of $T_\infty$.

For $k=1$, this gives the pattern $P_{2,1}$, which underlies the tiling-based upper-bound argument for classical domination. Figure~\ref{fig:mm:patron21} depicts $P_{2,1}$ on~$T_\infty$, with a pattern vertex at the origin $(0,0)$.

\begin{figure}[htbp]
    \centering
    
    \begin{subfigure}{0.4\textwidth}
        \centering
        \includegraphics[width=\linewidth]{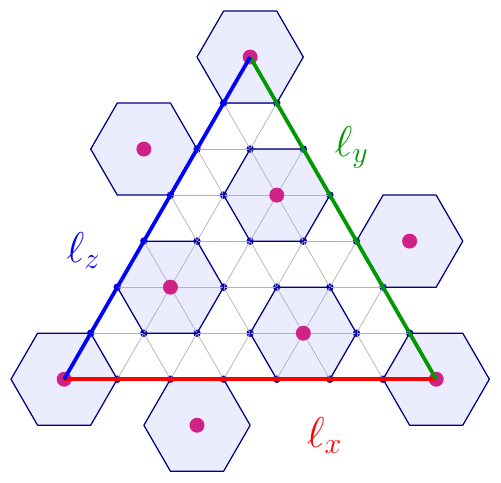}
        \caption{Sides $\ell_x,\ell_y,\ell_z$ and a dominating set for $T_7$ using vertices of~$P_{2,1}$.}
        \label{fig:mm:correccion1}
    \end{subfigure}
    \hfill
    \begin{subfigure}{0.55\textwidth}
        \centering
        \includegraphics[width=\linewidth]{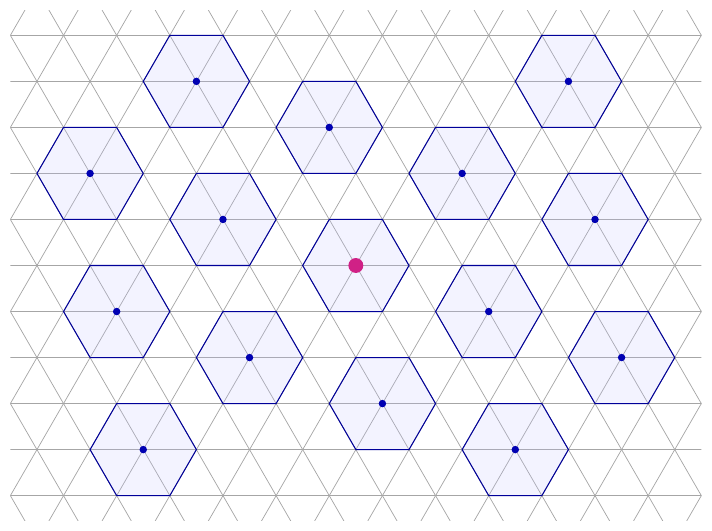}
        \caption{Pattern $P_{2,1}$ on $T_\infty$, with a pattern vertex at the origin $(0,0)$.}
        \label{fig:mm:patron21}
    \end{subfigure}
    
    \caption{Side notation, the pattern $P_{2,1}$, and its application to $T_7$.}
    \label{fig:mm:patron_y_dominacion}
\end{figure}

The pattern $P_{3,1}$ is obtained from the same construction when $k=2$ and underlies the corresponding argument for distance-$2$ domination. We begin by analyzing each of~\cite[Theorems~4.2 and~4.3]{harris_2020} in the case $\beta = 0$, showing that the stated bounds do not hold in this setting.

In \cite[Theorem 4.2]{harris_2020}, the authors cover $T_d$ with copies of $T_7$ using $P_{2,1}$ and argue as follows:

\begin{itemize}
  \item The vertex at the intersection of the sides $\ell_x$ and $\ell_z$ of $T_d$ is placed at $(0,0)$.
  \item The pattern $P_{2,1}$ contains the vertex at $(0,0)$.
  \item If a vertex belongs to the pattern, then all vertices at distance $7$ from it along the directions of the grid also belong to the pattern.
  \item The graph $T_d$ can be covered with $q^2$ copies of $T_7$.
  \item Each copy of $T_7$ is dominated by $9$ vertices of the pattern (Figure~\ref{fig:mm:correccion1}).
\end{itemize}

From this, they obtain the preliminary bound
\[
\gamma(T_d) \le 9q^2.
\]

Since vertices on shared internal edges of adjacent copies of $T_7$ are counted twice, Harris et al. subtract a correction term for this double counting. They argue that this correction equals $6q(q-1)$, since each shared side between two copies of $T_7$ contains four duplicated dominant vertices. The number of such internal sides in the decomposition of~$T_d$ is $\frac{3q(q-1)}{2}$, multiplying by $4$ yields $6q(q-1)$, leading to
\[
\gamma(T_d) \le 3q(q+2).
\]

This correction is valid for vertices lying on at most two internal edges of the decomposition, but it fails at vertices incident with six copies in the tiling. Figure~\ref{fig:mm:correccion2} shows a degree-$4$ vertex $u$ lying on two internal edges: in this situation, $u$ is counted three times and subtracted twice, leaving one occurrence, as intended. In contrast, Figure~\ref{fig:mm:correccion3} shows a degree-$6$ vertex $u$, which lies on six internal edges. Applying the same rule would subtract $u$ once for each incident internal edge, removing all occurrences of $u$ from the count. Consequently, $u$ (and the vertices it is supposed to dominate) are no longer accounted for, and the resulting total is underestimated.

\begin{figure}[htbp]
    \centering
    
    \begin{subfigure}{0.48\textwidth}
        \centering
        \includegraphics[width=\linewidth]{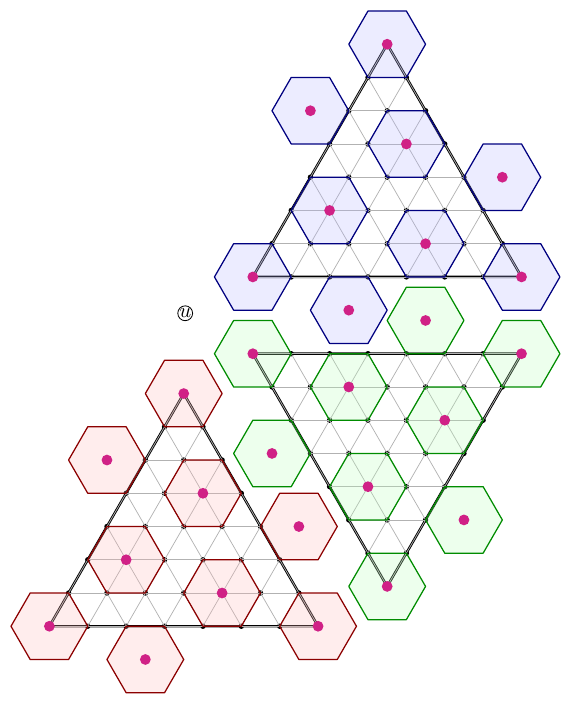}
        \caption{A vertex of degree $4$.}
        \label{fig:mm:correccion2}
    \end{subfigure}
    \hfill
    \begin{subfigure}{0.48\textwidth}
        \centering
        \includegraphics[width=\linewidth]{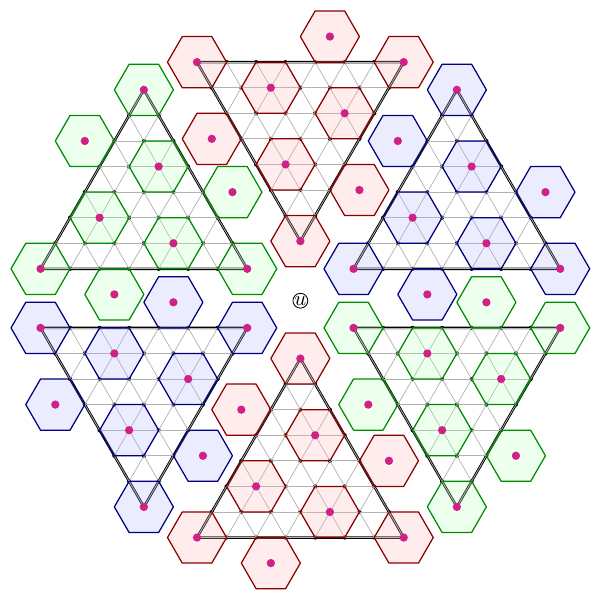}
        \caption{A vertex of degree $6$.}
        \label{fig:mm:correccion3}
    \end{subfigure}
    
    \caption{Application of the method to vertices of different degrees.}
    \label{fig:mm:correcciones_grados}
\end{figure}

The smallest instance where this issue appears is $T_{21}$. Harris et al.\ claim that~$\gamma(T_{21}) \le 45$. However, when the same construction is carried out consistently, the pattern $P_{2,1}$ induces a dominating set of size $46$ for $T_{21}$ (Figure~\ref{fig:mm:correccion4}). 

\begin{figure} [htbp]
  \begin{center}
    \begin{tabular}{c}
      \includegraphics{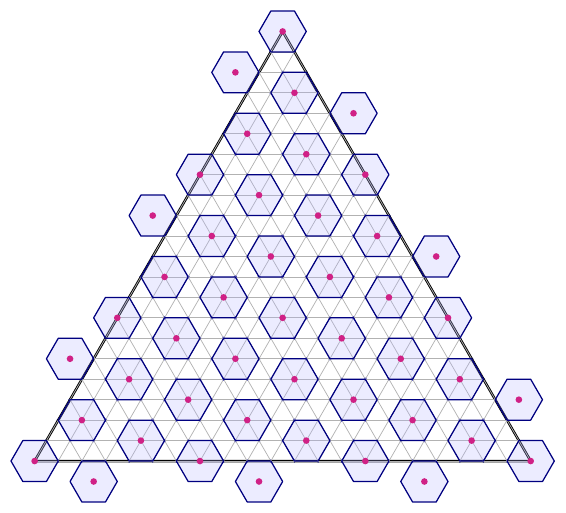}
      \\
    \end{tabular}
    \caption{Dominating set for $T_{21}$ using $46$ vertices of $P_{2,1}$.}
    \label{fig:mm:correccion4}
  \end{center}
\end{figure}

The same issue appears in the proof of Theorem~4.3 of~\cite{harris_2020}, where~$T_d$ is covered by copies of $T_{19}$. Vertices of degree $6$ again occur at points where six copies meet, producing the same invalid cancellation. The first failure arises in $T_{57}$, which contains such a configuration.

Finally, we obtain both an asymptotic and a concrete contradiction between \cite[Lemma~3.1]{harris_2020} and \cite[Theorems~4.2 and~4.3]{harris_2020}. For $k=1$ and $d=7q$, Lemma~3.1 yields
\[
\gamma(T_d)\ \ge\ \left\lceil \frac{(d+2)(d+1)}{14}\right\rceil
= 3.5q^2 + O(q),
\]
whereas Theorem~4.2 asserts
\[
\gamma(T_{7q}) \le 3q(q+2)=3q^2+O(q).
\]
Hence, for sufficiently large $d$, the lower bound must exceed the proposed upper bound. This already occurs at $d=63$ (i.e., $q=9$), where
$
\left\lceil \frac{65\cdot 64}{14}\right\rceil = 298
\quad\text{but}\quad
3\cdot 9\cdot 11 = 297,
$
so the claimed upper bound is strictly smaller than a valid lower bound, and the formula in Theorem~4.2 cannot be correct.

Similarly, for $k=2$ and $d=19q$, Lemma~3.1 gives
\[
\gamma_2(T_d)\ \ge\ \left\lceil \frac{(d+2)(d+1)}{38}\right\rceil
= 9.5q^2 + O(q),
\]
whereas Theorem~4.3 asserts
\[
\gamma_2(T_{19q}) \le 9q(q+1)=9q^2+O(q).
\]
Again, for sufficiently large $d$ the lower bound exceeds the proposed upper bound, and this already occurs at $d=285$ (i.e., $q=15$), where
$
\left\lceil \frac{287\cdot 286}{38}\right\rceil = 2161
\quad\text{but}\quad
9\cdot 15\cdot 16 = 2160.
$
Therefore, the formula in Theorem~4.3 cannot be correct either.

%% file: sections/2_generalization/generalization.tex
\section{Generalization of the upper bound for \texorpdfstring{$\gamma_k(T_d)$}{gamma_k(Td)}}\label{sec:generalization}

% 1. Probar que para cada T_d, con D_k = 3k^2+3k+1, si d = q * D_k, entonces V_k^in (conjunto de vértices internos que pertenecen a P_{k+1,1} ).
% 2. Probar cuántos vértices necesito por lado para cada k, serían 2k + 2, para completar la dominación. B_k
% 3. Probar que I_q es generalizable para cualquier k.
% 4. Probar el upper bound haciendo X + (2k+2)*3 - I_q

In this section, we generalize the corrected tiling argument of Harris et al.~\cite{harris_2020} in order to obtain an upper bound for $\gamma_k(T_d)$. For $k\geq 1$, we define
\[
D_k=3k^2+3k+1.
\]
This value is the distance between consecutive vertices of the pattern $P_{k+1,1}$ along any fixed direction of the triangular lattice. We first consider the case in which, for some integer $q\geq 1$,
\[
d=qD_k.
\]
We place $T_d$ in the infinite triangular lattice $T_\infty$ so that its three corner vertices belong to $P_{k+1,1}$. Under this placement, each side of $T_d$ is divided by vertices of the pattern into $q$ intervals of length $D_k$. Equivalently, $T_d$ can be decomposed into $q^2$ copies of $T_{D_k}$.

We use the following notation. Let $V_k^{\mathrm{in}}$ denote the number of vertices of $P_{k+1,1}$ that belong to a copy of $T_{D_k}$. 
Let $V_k^{\mathrm{ex}}$ denote the number of vertices of $P_{k+1,1}$ that do not belong to this copy of $T_{D_k}$ but are used to dominate vertices on its boundary.
Thus,
\[
V_k=V_k^{\mathrm{in}}+V_k^{\mathrm{ex}}
\]
is the total number of vertices of the pattern used to dominate one copy of $T_{D_k}$.

Moreover, let $B_k$ denote the number of pattern vertices counted in both copies when two adjacent copies of $T_{D_k}$ share a side in the decomposition of $T_d$. 
Finally, let $I_q$ denote the number of interior vertices of the decomposition that belong to $P_{k+1,1}$ and at which six copies of $T_{D_k}$ meet.

\input{sections/2_generalization/lem_internal_pattern_count}

\input{sections/2_generalization/lem_external_pattern_count}

\input{sections/2_generalization/lem_six_copy_intersections}

\input{sections/2_generalization/thm_general_upper_bound}

%% file: sections/2_generalization/lem_internal_pattern_count.tex
\begin{lemma}\label{lem:internal-pattern-count}
Let $D_k=3k^2+3k+1$. In a standard copy of $T_{D_k}$, the pattern $P_{k+1,1}$ contains exactly
\[
V_k^{\mathrm{in}}=\frac{D_k+5}{2}
\]
pattern vertices.
\end{lemma}

\begin{proof}

The proof follows the same Euler-characteristic counting principle underlying Pick's theorem~\cite{pick_1899}.
Consider a standard copy of $T_{D_k}$ whose three corner vertices belong to the triangular sublattice determined by $P_{k+1,1}$. This sublattice induces a triangulation of $T_{D_k}$. Let $V,E,F$ denote the numbers of vertices, edges, and faces of this triangulation, respectively, where $F$ includes the exterior face.

Each triangular cell of the induced sublattice has area $\frac{D_k\sqrt{3}}{4}$, whereas the area of $T_{D_k}$ is $\frac{D_k^2\sqrt{3}}{4}$. Hence the number of bounded triangular faces is $D_k$, and therefore $F=D_k+1$.

Since consecutive vertices of $P_{k+1,1}$ on each side are separated by distance $D_k$, each side of the standard copy of $T_{D_k}$ contributes exactly one boundary edge to this triangulation. 
Hence the exterior face has three boundary edges, and so $b=3$.
Counting edge incidences over the bounded triangular faces gives
\[
3(F-1)=2E-b,
\]
because each interior edge is incident with two bounded triangular faces, whereas each boundary edge is incident with exactly one. Thus
\[
E=\frac{3(F-1)+b}{2}
=
\frac{3D_k+3}{2}.
\]

By Euler's formula,
\[
V-E+F=2.
\]

Since the vertices of this triangulation are precisely the vertices of $P_{k+1,1}$ contained in the standard copy of $T_{D_k}$, we obtain
\begin{align*}
V_k^{\mathrm{in}}=V
&=E-F+2 \\
&=\frac{3D_k+3}{2}-(D_k+1)+2 \\
&=\frac{D_k+5}{2}.
\end{align*}

\end{proof}

%% file: sections/2_generalization/lem_external_pattern_count.tex
\begin{lemma}\label{lem:external-pattern-count}
For a copy of $T_{D_k}$, 
$
V_k^{\mathrm{ex}}=3k.
$
Moreover, if two copies of $T_{D_k}$ share a side, then the number of pattern vertices counted in both copies with respect to that shared side is
$
B_k=2k+2.
$
\end{lemma}

\begin{proof}
Consider a copy of $T_{D_k}$ and its side $\ell_x$.
The same argument applies to the other two sides.
Let $R(\ell_x)$ be the parallelogram formed by the rows parallel to $\ell_x$ that are at distance at most $k$ from $\ell_x$: the $k$ rows outside $T_{D_k}$, the side $\ell_x$ itself, and the $k$ rows inside $T_{D_k}$; see Figure~\ref{fig:gen:parallelogram}, which illustrates the case $k=2$. 
Thus, $R(\ell_x)$ contains $2k+1$ rows parallel to $\ell_x$.

By the placement of the pattern $P_{k+1,1}$, the two endpoints of $\ell_x$ belong to the pattern.
Moreover, each row of $R(\ell_x)$ has the same length as $\ell_x$, namely length $D_k$ in the corresponding direction.
By the periodicity of $P_{k+1,1}$, each of these rows contains either one or two pattern vertices.
The only row that contains two pattern vertices is the central row, namely $\ell_x$ itself, since its two endpoints belong to $P_{k+1,1}$.

Indeed, if a row different from $\ell_x$ contained two pattern vertices, then these vertices would have to be located at its endpoints.
However, each endpoint of such a row is at distance at most $k$ from one of the endpoints of $\ell_x$.
Since the endpoints of $\ell_x$ already belong to the pattern, and since consecutive vertices of $P_{k+1,1}$ in each principal direction are separated by distance $D_k$, no endpoint of such a row can also belong to the pattern.
Therefore, every row different from $\ell_x$ contains exactly one vertex of $P_{k+1,1}$.

Among the $2k+1$ rows of $R(\ell_x)$, the central row contributes the two endpoints of $\ell_x$, the $k$ exterior rows contribute $k$ pattern vertices outside $T_{D_k}$, and the $k$ interior rows contribute $k$ pattern vertices inside $T_{D_k}$.
In particular, for the side $\ell_x$ of $T_{D_k}$ there are $k$ exterior pattern vertices that dominate vertices on that side.
The same argument applies to the exterior rows associated with the other two sides of $T_{D_k}$, and these three sets of exterior rows are pairwise disjoint. Hence,
\[
V_k^{\mathrm{ex}}=3k.
\]

Now consider two adjacent copies of $T_{D_k}$ sharing a side.
By symmetry, we may assume that the shared side is parallel to $\ell_x$.
Applying the previous argument to the parallelogram associated with this shared side, the pattern vertices counted in both copies with respect to this side are the two endpoints of the side, the $k$ vertices on one side, and the $k$ vertices on the other side.
Therefore,
\[
B_k=2+k+k=2k+2.
\]
\end{proof}

\begin{figure}[htbp]
    \centering

    \begin{subfigure}{0.68\textwidth}
        \centering
        \includegraphics[width=0.85\linewidth]{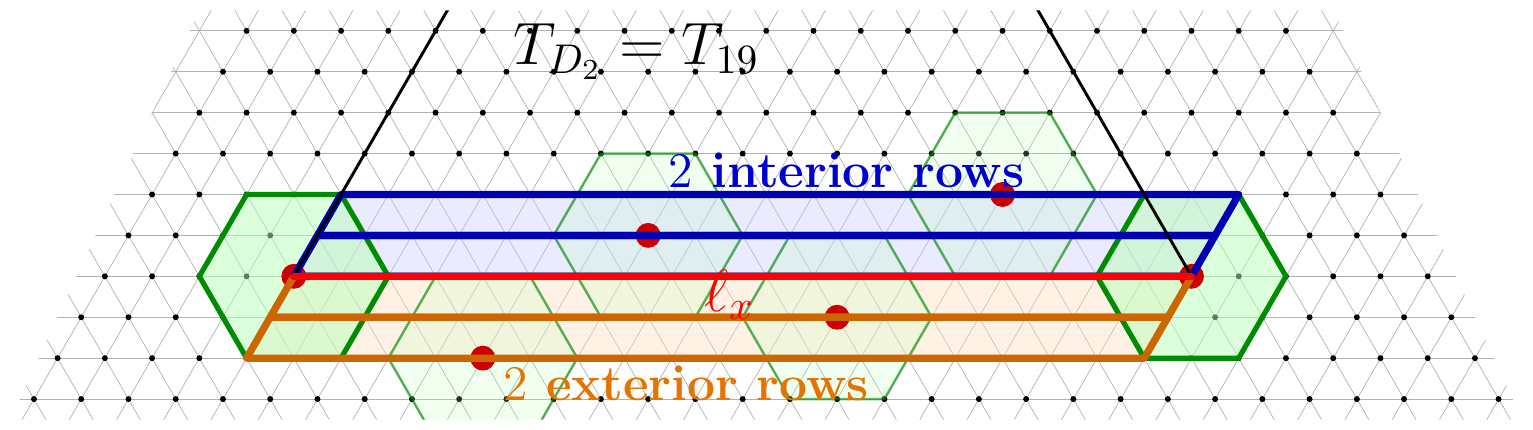}
        \caption{The parallelogram $R(\ell_x)$ associated with the side $\ell_x$, illustrated for $k=2$. It contains the two exterior rows, the side $\ell_x$, and the two interior rows used to count the boundary vertices; in general, $R(\ell_x)$ contains $2k+1$ rows.}
        \label{fig:gen:parallelogram}
    \end{subfigure}
    \hfill
    \begin{subfigure}{0.3\textwidth}
        \centering
        \includegraphics[width=\linewidth]{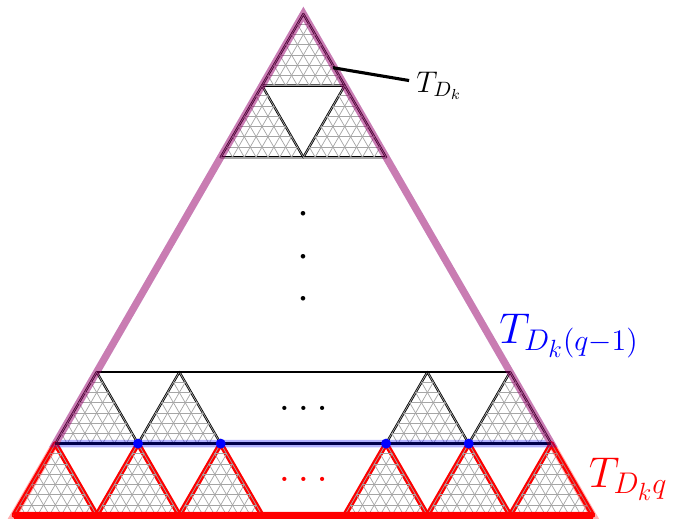}
        \caption{Schematic addition of a new row of copies of $T_{D_k}$ to pass from $T_{D_k(q-1)}$ to $T_{D_kq}$ in the induction for $I_q$.}
        \label{fig:gen:row-addition}
    \end{subfigure}
    
    \caption{Local configurations for the generalized tiling argument.}
    \label{fig:gen:local-configurations}
\end{figure}

%% file: sections/2_generalization/lem_six_copy_intersections.tex
\begin{lemma}\label{lem:six-copy-intersections}
Let $d=D_kq$, with $q\geq 1$.
Consider the decomposition of $T_d$ into $q^2$ copies of $T_{D_k}$ induced by the placement of the pattern $P_{k+1,1}$.
Then the number of interior vertices of $T_d$ that belong to $P_{k+1,1}$ and lie at the intersection of six copies of $T_{D_k}$ is
\[
I_q=\frac{(q-1)(q-2)}{2}.
\]
\end{lemma}

\begin{proof}
We proceed by induction on $q$.
For $q=1$, the decomposition consists of a single copy of $T_{D_k}$, so there are no interior vertices lying at the intersection of six copies of $T_{D_k}$.
Hence,
$
I_1=0=\frac{(1-1)(1-2)}{2}.
$
Assume that the result holds for $q-1$, with $q>1$. 
The graph $T_{D_kq}$ is obtained from $T_{D_k(q-1)}$ by adding a new row of copies of $T_{D_k}$ along the side $\ell_x$, as illustrated schematically in Figure~\ref{fig:gen:row-addition}.

By the placement of the pattern, the side $\ell_x$ of $T_{D_k(q-1)}$ is divided into $q-1$ intervals of length $D_k$ by vertices of $P_{k+1,1}$.
Therefore, this side contains $q$ pattern vertices, of which $q-2$ are not corners. When the new row of copies of $T_{D_k}$ is added, precisely these $q-2$ vertices cease to lie on the boundary and become interior vertices of $T_{D_kq}$ where six copies of $T_{D_k}$ meet.

Thus, the number of new vertices with this property is $q-2$. By the induction hypothesis,
\begin{align*}
I_q
&= I_{q-1}+(q-2) \\
&= \frac{(q-2)(q-3)}{2}+(q-2) \\
&= \frac{(q-1)(q-2)}{2}.
\end{align*}
This completes the induction.
\end{proof}

%% file: sections/2_generalization/thm_general_upper_bound.tex
\general*

\begin{proof}
First consider the case $\beta=0$.
Then $d=D_kq$.
By Lemma~\ref{lem:internal-pattern-count}, the number of vertices of the pattern $P_{k+1,1}$ contained in a copy of $T_{D_k}$ is
\[
V_k^{\mathrm{in}}=\frac{D_k+5}{2}.
\]
Moreover, by Lemma~\ref{lem:external-pattern-count}, the number of exterior pattern vertices needed to dominate the boundary of a copy of $T_{D_k}$ is $V_k^{\mathrm{ex}}=3k$. Therefore,
\[
V_k
=
V_k^{\mathrm{in}}+V_k^{\mathrm{ex}}
=
\frac{D_k+5}{2}+3k.
\]
Since $D_k=3k^2+3k+1$, we obtain
\[
V_k
=
\frac{3k^2+3k+6}{2}+3k
=
\frac{3k^2+9k+6}{2}.
\]

Now, since $T_d=T_{D_kq}$ can be decomposed into $q^2$ copies of $T_{D_k}$, considering each copy separately gives the initial count $V_kq^2$.
However, when two adjacent copies share a side, some pattern vertices are counted in both local counts. By Lemma~\ref{lem:external-pattern-count}, for each shared side the number of such vertices is
\[
B_k=2k+2.
\]

The decomposition of $T_d$ into $q^2$ copies of $T_{D_k}$ has
\[
\frac{3q(q-1)}{2}
\]
shared sides.
Hence, we subtract
\[
B_k\frac{3q(q-1)}{2}.
\]

This subtraction leaves one occurrence of every pattern vertex that is counted in two or three local counts, since such a vertex is subtracted through the correction terms of one or two shared sides, respectively. The exceptional case occurs at an interior vertex where six copies meet: such a vertex is counted six times and is subtracted once for each of the six incident shared sides, leaving no occurrence. Hence these vertices must be added back once. By Lemma~\ref{lem:six-copy-intersections}, their number is
\[
I_q=\frac{(q-1)(q-2)}{2}.
\]
Thus,
\[
\gamma_k(T_d)
\leq
V_kq^2
-
B_k\frac{3q(q-1)}{2}
+
I_q.
\]
Substituting $V_k=\frac{3k^2+9k+6}{2}$, $B_k=2k+2$, and $I_q=\frac{(q-1)(q-2)}{2}$, we obtain
\begin{align*}
\gamma_k(T_d)
&\leq
\frac{3k^2+9k+6}{2}q^2
-
(2k+2)\frac{3q(q-1)}{2}
+
\frac{(q-1)(q-2)}{2} \\
&=
\frac{3k^2+9k+6}{2}q^2
-
3(k+1)q(q-1)
+
\frac{q^2-3q+2}{2} \\
&=
\frac{(3k^2+3k+1)q^2+(6k+3)q+2}{2}.
\end{align*}
Since $D_k=3k^2+3k+1$, it follows that
\[
\gamma_k(T_d)
\leq
\frac{D_kq^2+(6k+3)q+2}{2}.
\]

Now suppose that $\beta>0$. Then
\[
d=D_kq+\beta<D_k(q+1).
\]
In this case, we apply the previous construction to the triangle $T_{D_k(q+1)}$.
Since $T_d$ is contained in this larger triangle, the same construction provides a distance-$k$ dominating set for $T_d$ with no more vertices than those used for $T_{D_k(q+1)}$.
Therefore, applying the case $\beta=0$ with $q+1$ in place of $q$, we obtain
\[
\gamma_k(T_d)
\leq
\frac{D_k(q+1)^2+(6k+3)(q+1)+2}{2}.
\]
This completes the proof.
\end{proof}

%% file: sections/3_improvement/improvement.tex
\section{Improvement of the upper bound for the domination number of \texorpdfstring{$T_d$}{Td}}\label{sec:improvement}

In this section, we improve the corrected upper bound for the domination number of triangular matchstick graphs in the case $k=1$. More precisely, for $d=7q+\beta$, with $q\geq 1$ and $0\leq \beta \leq 6$, we construct a dominating set of $T_d$ whose size improves the bound obtained from the corrected version of the argument of Harris et al.

\improvementkone*

\begin{proof}
The proof is carried out in three phases that apply to all values of $\beta$. 
In Phase~I, we extend a dominating set of $T_{7q}$ to dominate $T_{7q+\beta}$ and obtain an initial upper bound for $\gamma(T_d)$.
In Phase~II, we translate the graph within the infinite triangular grid while preserving the initial distribution of the vertices of $P_{2,1}$, so that the upper bound on $\gamma(T_d)$ is maintained and may decrease in some cases.
In Phase~III, we perform local optimizations that further decrease the size of the dominating set. The final bounds follow by combining the contributions of these three phases for each value of $\beta$.

\begin{figure}[htbp]
    \centering
    
    % beta 1 y 2
    \begin{subfigure}{0.42\textwidth}
        \centering
        \includegraphics[width=\linewidth]{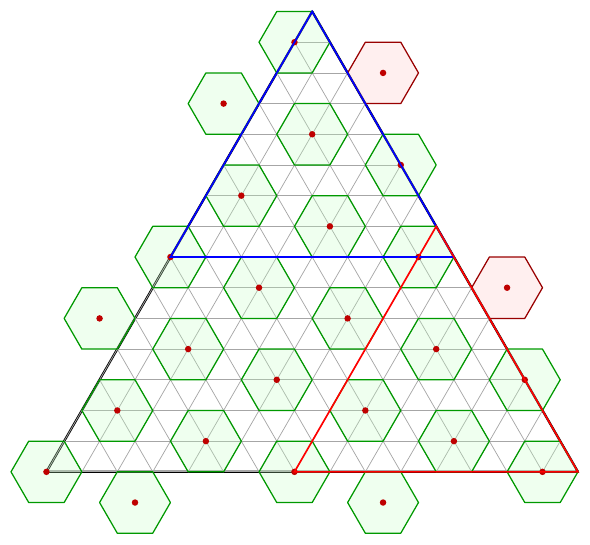}
        \caption{$\beta = 1$.}
    \end{subfigure}
    \hfill
    \begin{subfigure}{0.42\textwidth}
        \centering
        \includegraphics[width=\linewidth]{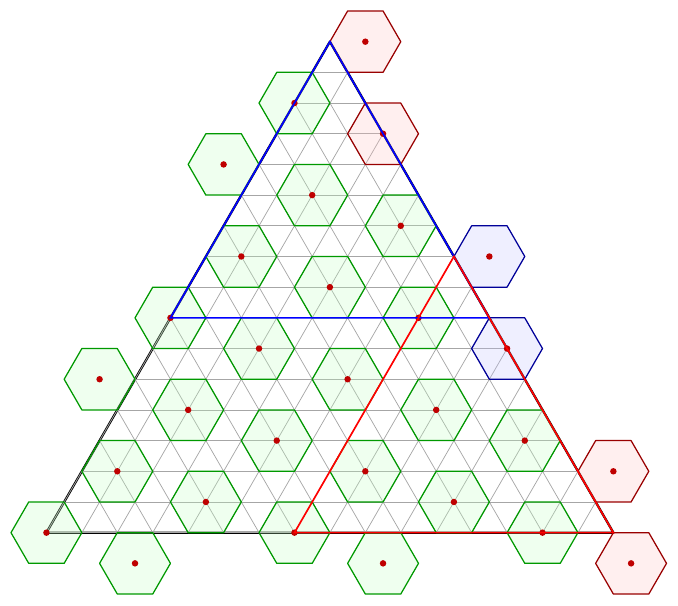}
        \caption{$\beta = 2$.}
    \end{subfigure}
    
    % beta 3 y 4
    \begin{subfigure}{0.45\textwidth}
        \centering
        \includegraphics[width=\linewidth]{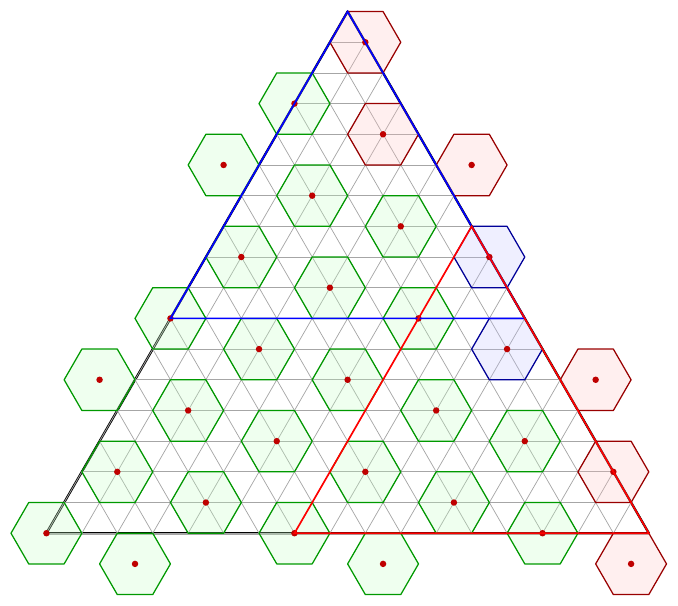}
        \caption{$\beta = 3$.}
    \end{subfigure}
    \hfill
    \begin{subfigure}{0.45\textwidth}
        \centering
        \includegraphics[width=\linewidth]{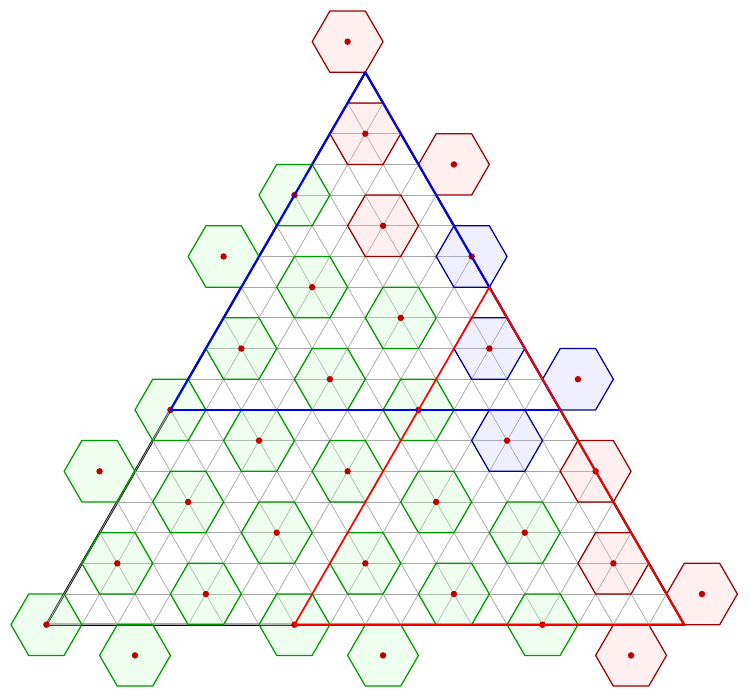}
        \caption{$\beta = 4$.}
    \end{subfigure}

    % beta 5 y 6
    \begin{subfigure}{0.47\textwidth}
        \centering
        \includegraphics[width=\linewidth]{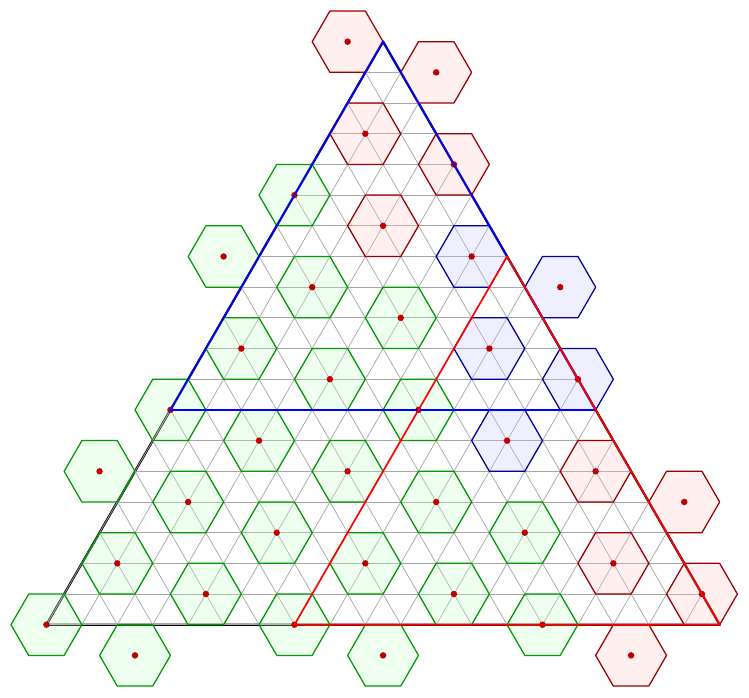}
        \caption{$\beta = 5$.}
    \end{subfigure}
    \hfill
    \begin{subfigure}{0.47\textwidth}
        \centering
        \includegraphics[width=\linewidth]{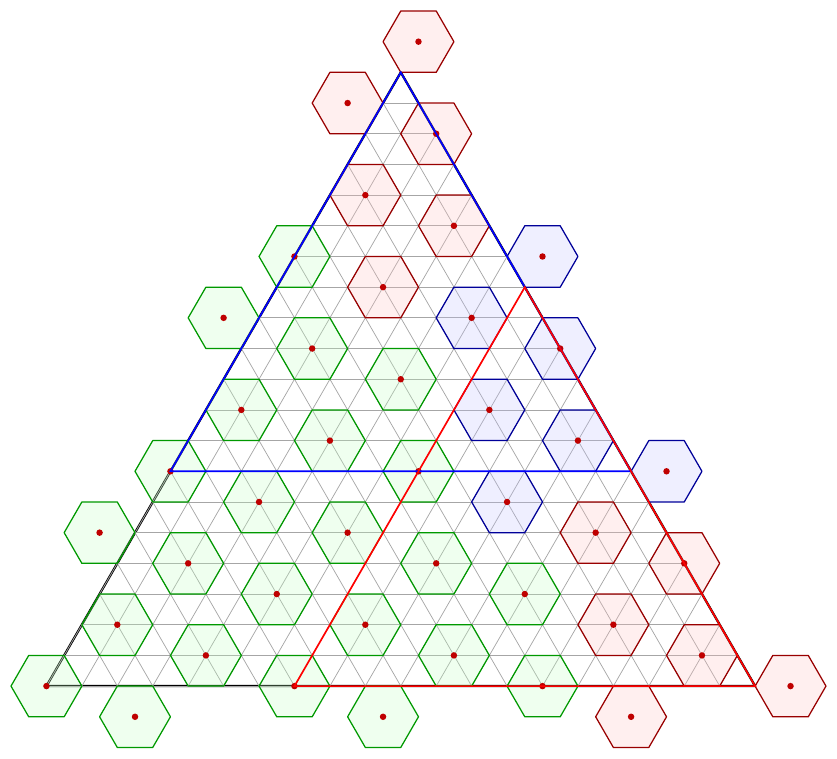}
        \caption{$\beta = 6$.}
    \end{subfigure}
    
    \caption{Phase~I (construction). Extension of a dominating set from $T_{7q}$ to~$T_{7q+\beta}$ for $q=2$ and $1 \le \beta \le 6$. The vertices with green hexagons form a dominating set of $T_{7q}$, while the red and blue vertices represent the additional vertices required to dominate the attached copies of $T_{7+\beta}$; the blue vertices are shared by two adjacent copies.}
    \label{fig:teorema_espec:fase1}
\end{figure}

%%%%%%%%%% FASE 1
\paragraph{Phase I: Construction of $T_{7q+\beta}$ from $T_{7q}$.}

Recall that $d=7q+\beta$. 
The graph $T_d$ contains a unique subgraph isomorphic to $T_{7(q-1)}$ that shares the vertex at the intersection of the sides~$\ell_x$ and $\ell_z$ of $T_d$. 
Along the side $\ell_y$ of this subgraph, the vertices of the pattern $P_{2,1}$ occur periodically with a period of $7$, including both endpoints of the side; consequently, there are exactly $q$ such vertices on this side.

Each of these vertices lies at the intersection of the sides $\ell_x$ and $\ell_z$ of a copy of $T_{7+\beta}$ attached to $T_{7(q-1)}$, and the vertices of $P_{2,1}$ that dominate each such copy form the same pattern depending only on $\beta$. Adjacent copies may share some dominating vertices, so the total number of additional vertices depends on $\beta$ and grows linearly with $q$.

For $1\le \beta\le 6$, this construction yields an initial upper bound on $\gamma(T_d)$ obtained by adding the vertices required to dominate the attached copies of $T_{7+\beta}$. These vertices are indicated in Figure~\ref{fig:teorema_espec:fase1}, where the green vertices correspond to those already present in a dominating set of~$T_{7q}$, while the red and blue vertices belong to the added set, with the blue vertices being those shared by two adjacent copies of $T_{7+\beta}$. When $\beta=0$, no additional copies are attached, and the bound follows directly from Theorem~\ref{thm:general-upper-bound} applied with $k=1$, since $D_1=7$. Therefore, for all $0\le \beta\le 6$, Phase~I gives the following initial estimate:

\[
\gamma(T_d) \le 3.5q^2 + (\beta+4.5)q + (\beta-1) + a_\beta,
\]
where
\[
a_\beta =
\begin{cases}
1, & \beta \in \{1,3\},\\
2, & \beta \in \{0,2,4,5\},\\
3, & \beta = 6.
\end{cases}
\]

%%%%%%%%%% FASE 2

\paragraph{Phase II: Translation.}

To optimize the number of dominating vertices, we translate the graph $T_d$ within the infinite triangular grid $T_\infty$ while preserving the distribution of the pattern $P_{2,1}$. A \emph{translation} consists of shifting the graph rigidly without altering the relative positions of the vertices of $P_{2,1}$; hence the resulting set remains a valid dominating set. In particular, we apply a translation in the direction parallel to the side $\ell_z$, as illustrated in Figure~\ref{fig:teorema_espec:traslacion}. 

\begin{figure}[htbp]
    \centering
    \includegraphics[width=0.4\linewidth]{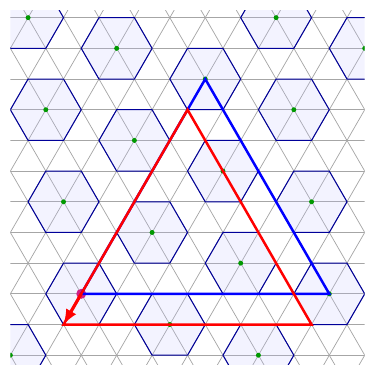}
    \caption{Translation of the graph $T_7$ (from blue to red) on $T_\infty$.}
    \label{fig:teorema_espec:traslacion}
\end{figure}

Under this translation, some dominating vertices required in Phase~I are no longer needed, while others must be added along the new boundary. In Figure~\ref{fig:teorema_espec:fase2}, the red vertices indicate those added after the translation, and the blue vertices indicate those that are no longer required. This produces a revised upper bound on $\gamma(T_d)$ obtained after the translation, leading to the following estimate:
\[
\gamma(T_d) \le 3.5q^2 + (\beta+4.5)q + b_\beta,
\]
where
\[
b_\beta =
\begin{cases}
1, & \beta \in \{0,1\},\\
2, & \beta = 2,\\
3, & \beta \in \{3,4\},\\
6, & \beta = 5,\\
7, & \beta = 6.
\end{cases}
\]

\begin{figure}[htbp]
    \centering
        
    % beta 0 1 2
    \begin{subfigure}{0.28\textwidth}
        \centering
        \includegraphics[width=\linewidth]{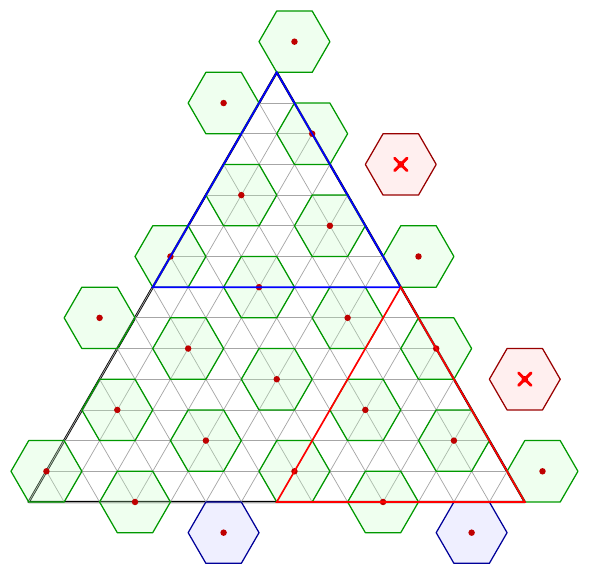}
        \caption{$\beta = 0$.}
    \end{subfigure}
    \hfill
    \begin{subfigure}{0.30\textwidth}
        \centering
        \includegraphics[width=\linewidth]{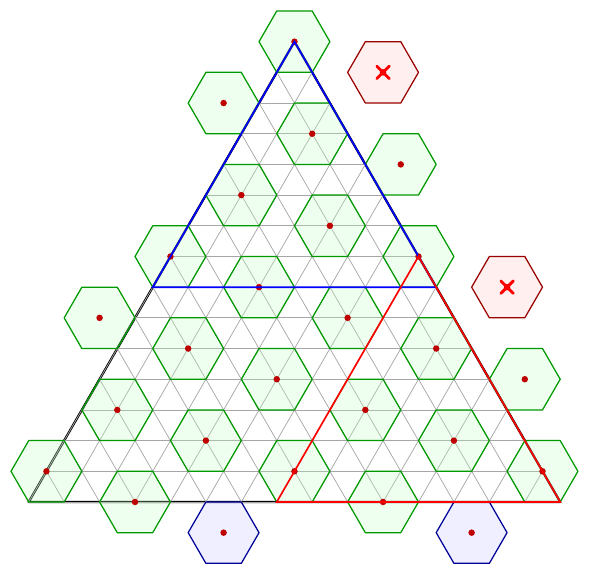}
        \caption{$\beta = 1$.}
    \end{subfigure}
    \hfill
    \begin{subfigure}{0.32\textwidth}
        \centering
        \includegraphics[width=\linewidth]{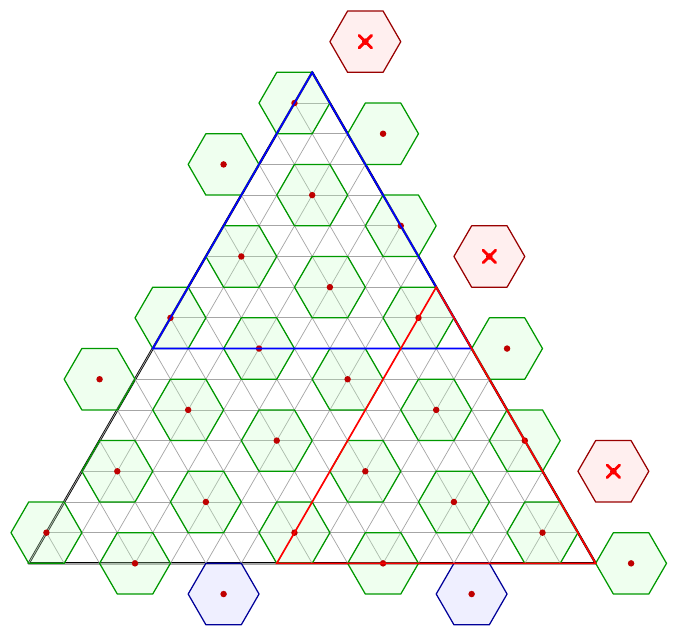}
        \caption{$\beta = 2$.}
    \end{subfigure}
    
    % beta 3 4
    \begin{subfigure}{0.48\textwidth}
        \centering
        \includegraphics[width=\linewidth]{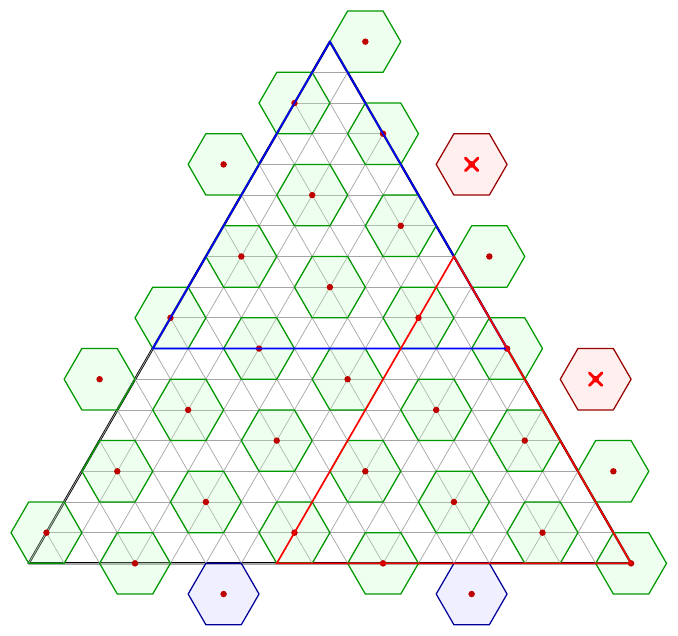}
        \caption{$\beta = 3$.}
    \end{subfigure}
    \hfill
    \begin{subfigure}{0.48\textwidth}
        \centering
        \includegraphics[width=\linewidth]{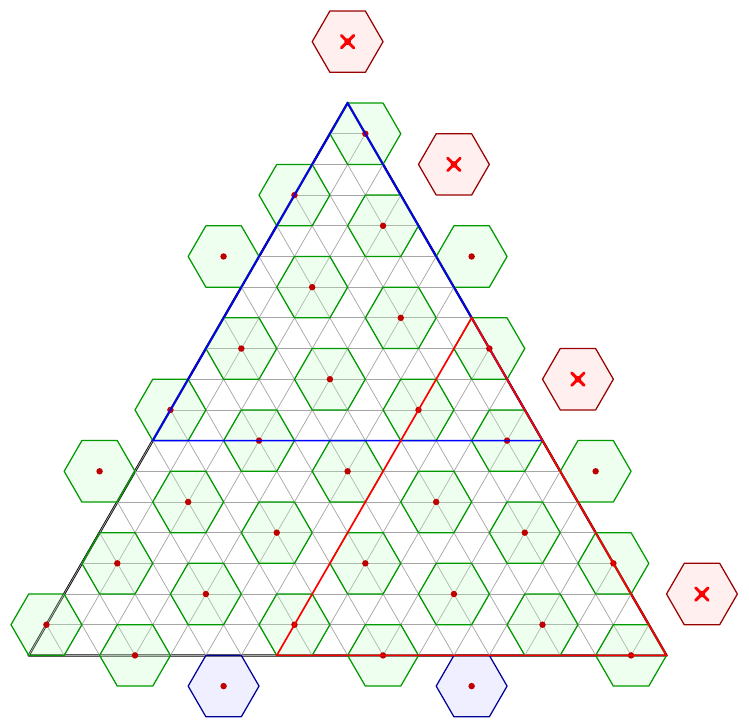}
        \caption{$\beta = 4$.}
    \end{subfigure}

    % beta 5 6
    \begin{subfigure}{0.48\textwidth}
        \centering
        \includegraphics[width=\linewidth]{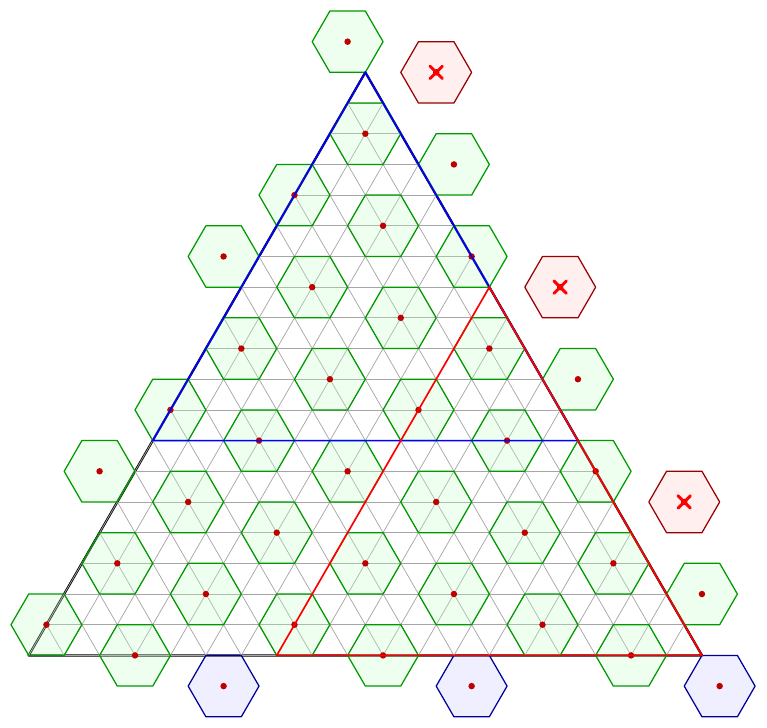}
        \caption{$\beta = 5$.}
    \end{subfigure}
    \hfill
    \begin{subfigure}{0.48\textwidth}
        \centering
        \includegraphics[width=\linewidth]{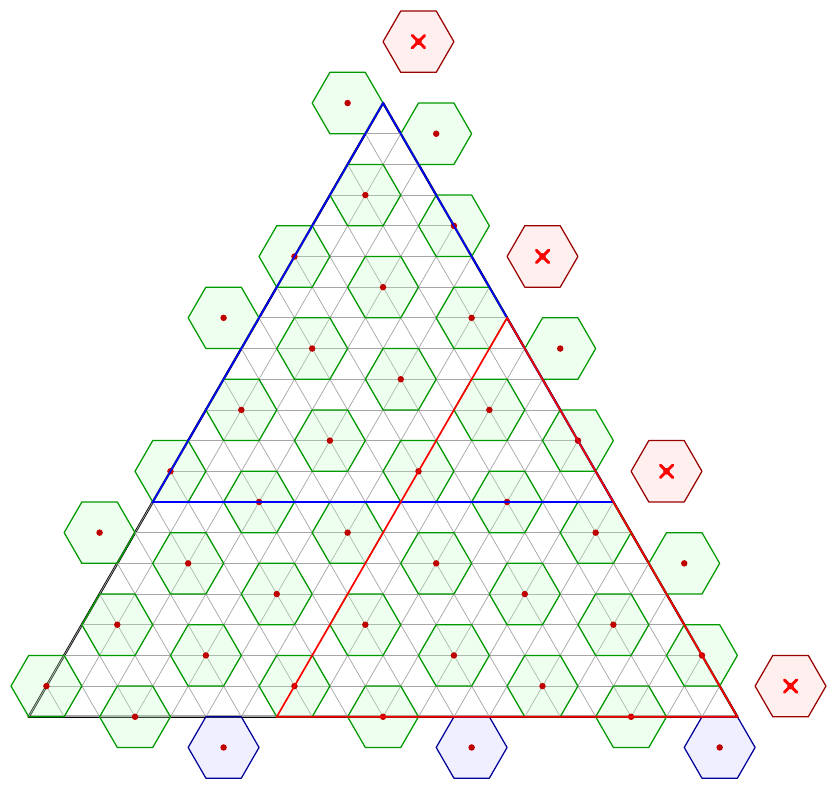}
        \caption{$\beta = 6$.}
    \end{subfigure}
    
    \caption{Phase~II (translation) applied to $T_{7q+\beta}$ for $q=2$ and $0 \le \beta \le 6$. All graphs are translated in the same direction within the infinite triangular grid while preserving the pattern $P_{2,1}$. The red vertices indicate dominating vertices that must be added after the translation, whereas the blue vertices indicate those that are no longer required.}
    \label{fig:teorema_espec:fase2}
\end{figure}

%%%%%%%%%% FASE 3

\paragraph{Phase III: Local optimization.}

In $T_{7q+\beta}$, denote by $T_{7q+\beta}^{U}$ the unique subgraph isomorphic to $T_{7+\beta}$ that contains the vertex at the intersection of the sides $\ell_y$ and $\ell_z$ of $T_{7q+\beta}$, and by $T_{7q+\beta}^{L}$ the unique subgraph isomorphic to $T_{7+\beta}$ that contains the vertex at the intersection of the sides $\ell_x$ and $\ell_y$ of $T_{7q+\beta}$.

In the final phase, we perform local optimizations inside one or both of the subgraphs $T_{7q+\beta}^{U}$ and $T_{7q+\beta}^{L}$ obtained from the translation. These optimizations consist of replacing certain subsets of dominating vertices of $P_{2,1}$ with smaller subsets that still dominate all vertices of $T_d$. More precisely, groups of $x$ dominating vertices are replaced by groups of $y$ vertices with $y<x$, without losing the domination property. The replacements depend only on $\beta$ and are preserved as $q$ increases.

The vertices removed and those added are indicated in Figure~\ref{fig:teorema_espec:fase3}. This yields the final estimate
\[
\gamma(T_d) \le 3.5q^2 + (\beta + 4.5)q + \beta - 1.
\]
This completes the proof of Theorem~\ref{teorema:upperbound_k1}.

\end{proof}

\begin{figure}[htbp]
    \centering
    
    % beta 0 1 2
    \begin{subfigure}{0.28\textwidth}
        \centering
        \includegraphics[width=\linewidth]{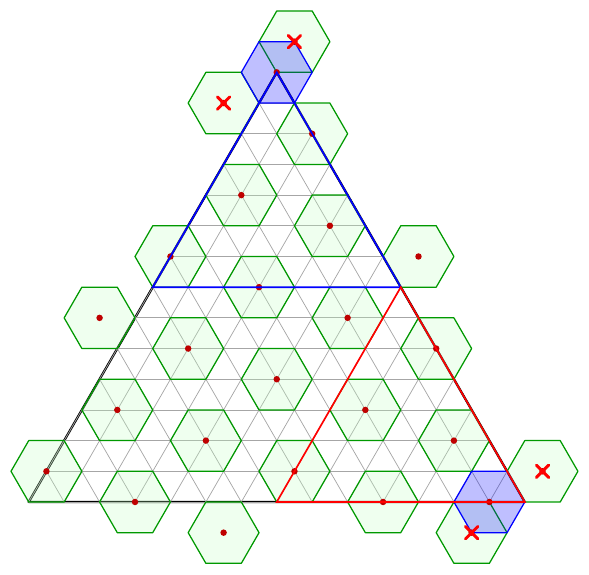}
        \caption{$\beta = 0$.}
    \end{subfigure}
    \hfill
    \begin{subfigure}{0.30\textwidth}
        \centering
        \includegraphics[width=\linewidth]{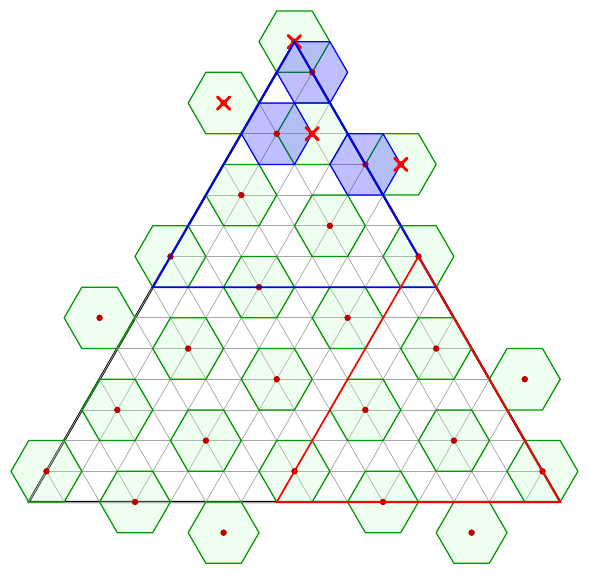}
        \caption{$\beta = 1$.}
    \end{subfigure}
    \hfill
    \begin{subfigure}{0.32\textwidth}
        \centering
        \includegraphics[width=\linewidth]{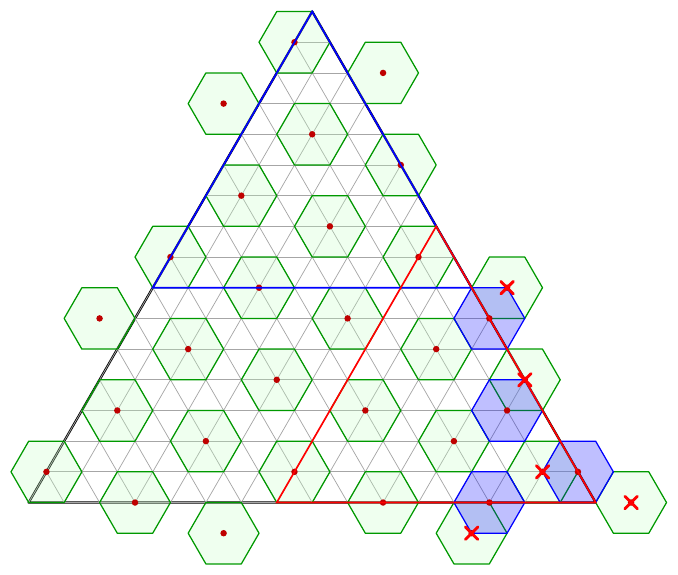}
        \caption{$\beta = 2$.}
    \end{subfigure}
    
    % beta 3 4
    \begin{subfigure}{0.48\textwidth}
        \centering
        \includegraphics[width=\linewidth]{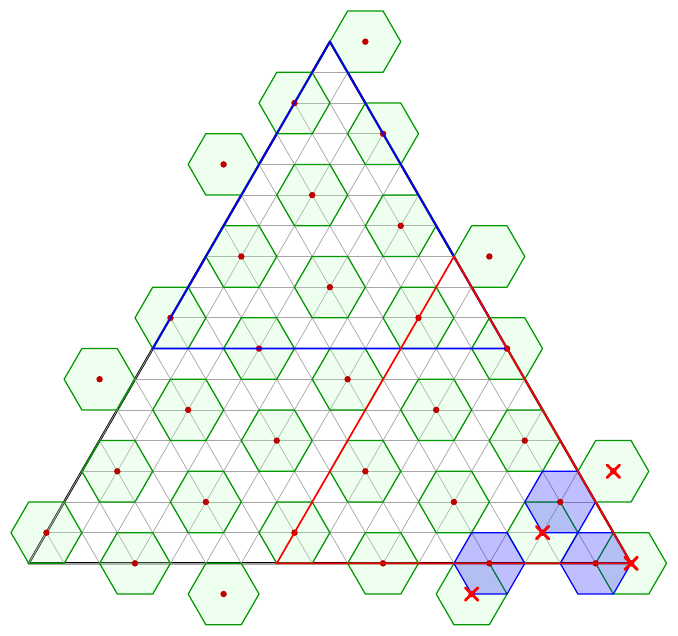}
        \caption{$\beta = 3$.}
    \end{subfigure}
    \hfill
    \begin{subfigure}{0.48\textwidth}
        \centering
        \includegraphics[width=\linewidth]{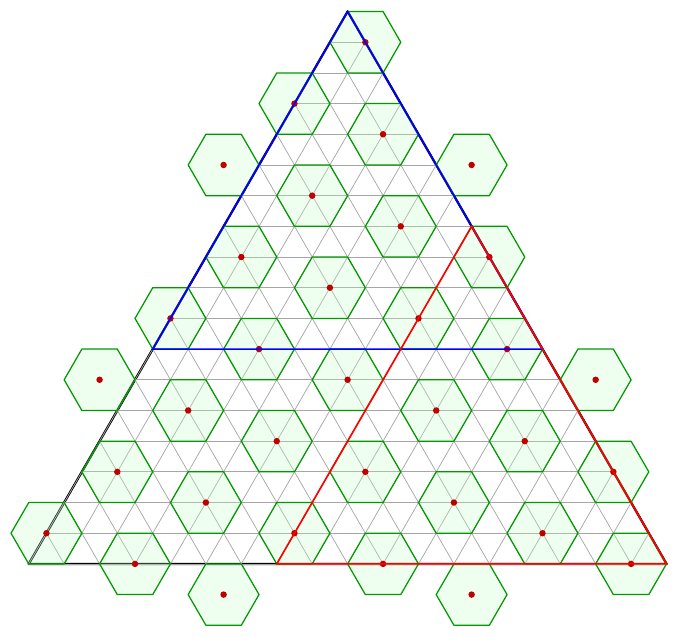}
        \caption{$\beta = 4$.}
    \end{subfigure}

    % beta 5 6
    \begin{subfigure}{0.48\textwidth}
        \centering
        \includegraphics[width=\linewidth]{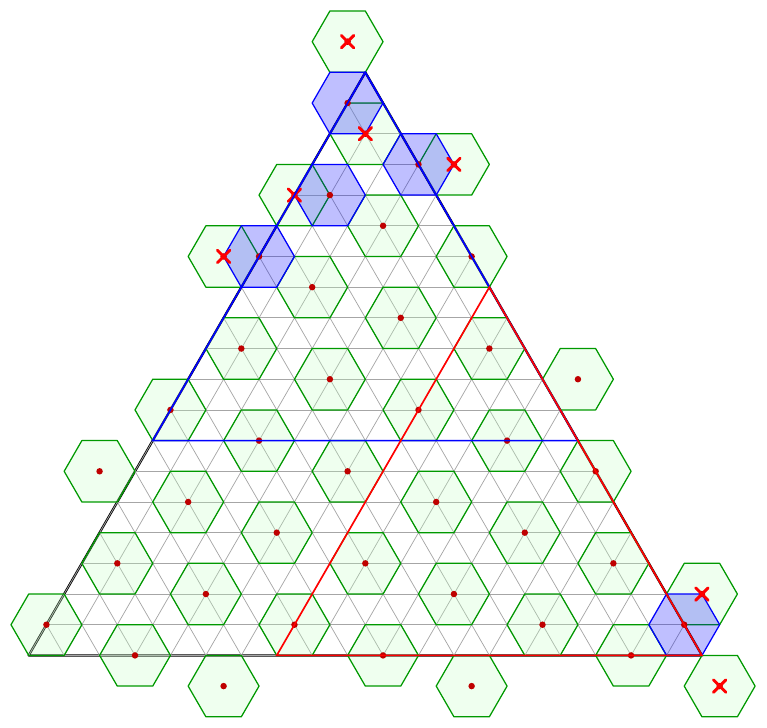}
        \caption{$\beta = 5$.}
    \end{subfigure}
    \hfill
    \begin{subfigure}{0.48\textwidth}
        \centering
        \includegraphics[width=\linewidth]{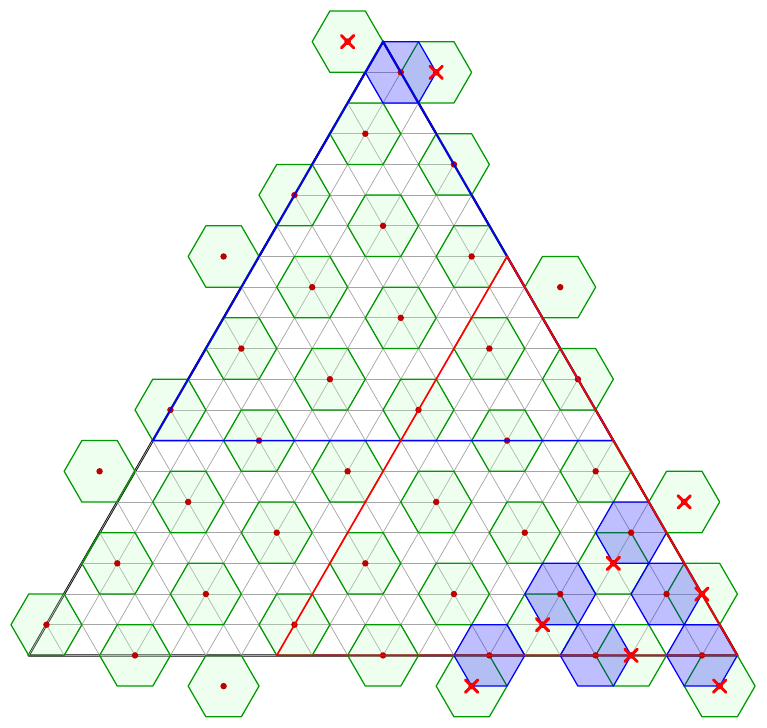}
        \caption{$\beta = 6$.}
    \end{subfigure}
    
    \caption{Phase~III (local optimization) applied to the translated graphs $T_{7q+\beta}$ for $0 \le \beta \le 6$. Certain dominating vertices of $P_{2,1}$ are replaced by smaller equivalent sets that still dominate the graph. The vertices marked with $\times$ are removed from the dominating set, while the shaded regions indicate the locations where the replacements occur.}
    \label{fig:teorema_espec:fase3}
\end{figure}

%% file: sections/4_radius/radius.tex
\section{Radius of triangular matchstick graphs}\label{sec:radius}

In this section, we determine the radius of $T_d$ and use it to obtain a threshold for which a single vertex distance-$k$ dominates the whole graph.

\radius*

\begin{proof}

It suffices to show that the following recurrence holds, since the sequence $\left\lceil 2d/3\right\rceil$ satisfies the same recurrence and the same initial values:
\[
r_d =
\begin{cases}
d & \text{if } d < 3,\\
r_{d-3}+2 & \text{if } d \ge 3.
\end{cases}
\]

We use coordinates $(x,y,z)$ for vertices of $T_d$ satisfying
$0 \le x,y,z \le d$ and~$y~=~x~+~z$, which uniquely identify each vertex (see also \cite{bose_power_2020}).
Adjacent vertices differ by $\pm1$ in exactly two coordinates and coincide in the third; consequently, each vertex has at most six neighbors arranged along three principal directions (Figure~\ref{fig:mm:lema1}). The corner vertices are $(0,0,0)$, $(d,d,0)$, and $(0,d,d)$, and each pair is at distance $d$. 
The three \emph{sides} of $T_d$ are the shortest paths between the corner vertices $(0,0,0)$, $(d,d,0)$, and $(0,d,d)$, denoted by $\ell_x$, $\ell_y$, and $\ell_z$. 
We call $(0,0,0)$ the \emph{origin}, and the side joining $(d,d,0)$ and $(0,d,d)$ the \emph{opposite side}.

\begin{figure}[htbp]
    \centering
    
    \begin{subfigure}{0.35\textwidth}
        \centering
        \includegraphics[width=\linewidth]{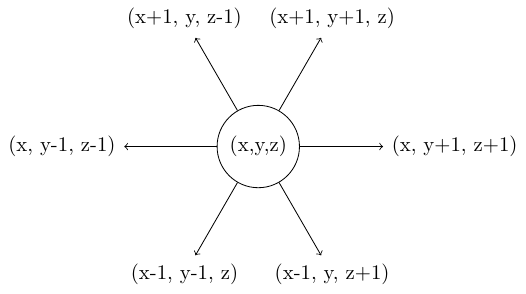}
        \caption{Coordinates of the neighbors of a vertex~$(x,y,z)$ of $T_d$.}
        \label{fig:mm:lema1}
    \end{subfigure}
    \hfill
    \begin{subfigure}{0.6\textwidth}
        \centering
        \includegraphics[width=0.75\linewidth]{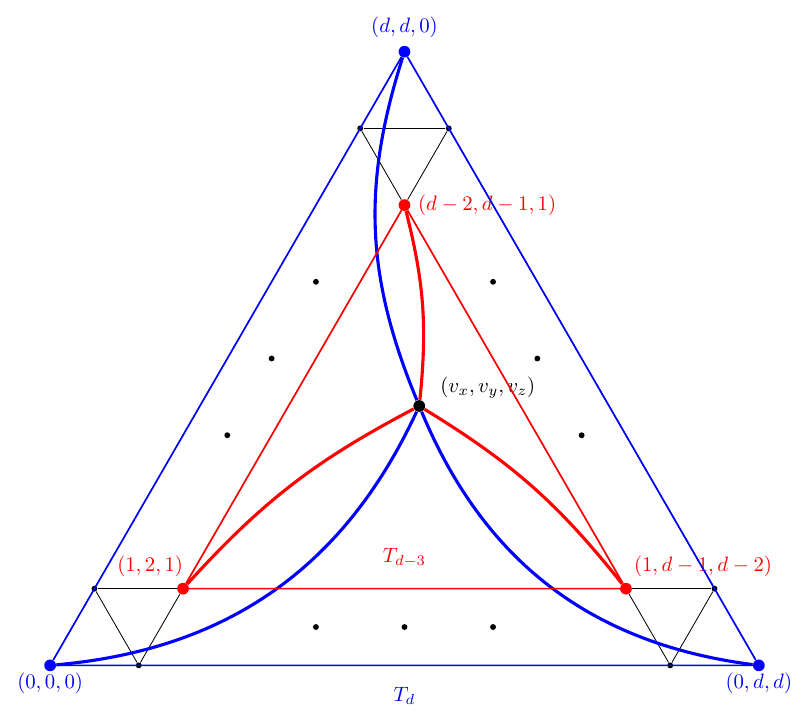}
        \caption{Construction of $T_d$ from $T_{d-3}$ by adding $3$ paths parallel to each of its sides.}
        \label{fig:mm:teo3:2}
    \end{subfigure}
    
    \caption{Coordinate system in $T_d$ and the construction of $T_d$ from $T_{d-3}$.}
    \label{fig:mm:coordsistema}
\end{figure}

\begin{claim}\label{claim:mindist}
If $u=(x_u,y_u,z_u)$ and $v=(x_v,y_v,z_v)$ are vertices of $T_d$, then
\[
dist(u,v)=\frac{|x_u-x_v|+|y_u-y_v|+|z_u-z_v|}{2}.
\]
\end{claim}

\begin{proof}
Let $P = u_0,\dots,u_m$ be a $u$--$v$ path with $u_0 = u$ and $u_m = v$, and define
\[
\phi_{ij}=|x_{u_j}-x_{u_i}|+|y_{u_j}-y_{u_i}|+|z_{u_j}-z_{u_i}|.
\]
Since adjacent vertices differ by $\pm1$ in two coordinates and coincide in the third,~
$\phi_{i(i+1)}~=~2$.
By the triangle inequality,
\[
\phi_{0m}\le\sum_{i=0}^{m-1}\phi_{i(i+1)}=2m,
\]
so every $u$--$v$ path has length at least $\frac{\phi_{0m}}{2}$.

For the reverse inequality, if $u\neq v$ there exists a neighbor $u'$ of $u$ such that
\[
|x_{u'}-x_v|+|y_{u'}-y_v|+|z_{u'}-z_v|
=
|x_u-x_v|+|y_u-y_v|+|z_u-z_v|-2.
\]
Iterating this construction yields a $u$--$v$ path of length $\frac{\phi_{0m}}{2}$.
\end{proof}

\begin{claim}\label{claim:ecc}
Let $u$ be a vertex of $T_d=(V,E)$, and let $v'$ be a corner vertex whose distance from $u$ is maximal. Then
$ecc(u) = dist(u,v')$.
\end{claim}

\begin{proof}
Let $f(x,y,z)=|x_u-x|+|y_u-y|+|z_u-z|$.
If $v$ is not a corner vertex, then $v$ has a neighbor $w$ such that
$f(w)\ge f(v)$. Since $T_d$ is finite, iterating this argument produces a corner vertex~$v'$
with $f(v')\ge f(v)$.
\end{proof}

We now continue with the proof of Theorem \ref{teo:radio}.
We proceed by induction on $d$. The cases $d\le2$ follow by inspection.
Assume $d\ge3$.
Figure~\ref{fig:mm:teo3:2} illustrates the construction of $T_d$ from $T_{d-3}$ by adding three outer strips, one parallel to each side.
Let $c_1,c_2,c_3$ be the corner vertices of $T_d$ and
$c_1',c_2',c_3'$ the corresponding corners of $T_{d-3}$.
For $v\in T_{d-3}$ we will prove that
\[
ecc(v)=ecc'(v)+2,
\]
where $ecc'$ denotes eccentricity in $T_{d-3}$.

By Claim~\ref{claim:ecc}, it suffices to show
that
$dist(v,c_i)=dist(v,c_i')+2.
$
We verify this for $c_1=(0,0,0)$ and $c_1'=(1,2,1)$; the other cases
are analogous.
By Claim~\ref{claim:mindist},
$
dist(v,c_1)=\frac{|v_x|+|v_y|+|v_z|}{2}
$
and
$
dist(v,c_1')=\frac{|v_x-1|+|v_y-2|+|v_z-1|}{2}.
$

Since $v\in T_{d-3}$ implies $v_x,v_z\ge1$ and $v_y=v_x+v_z\ge2$,
we obtain $|v_x-1|=v_x-1$, $|v_y-2|=v_y-2$, and $|v_z-1|=v_z-1$, so
$
dist(v,c_1')=dist(v,c_1)-2
$
and therefore
$
ecc(v)=ecc'(v)+2.
$
Moreover, Figure~\ref{fig:mm:teo3:2} shows that every vertex of $T_d$ outside $T_{d-3}$ lies in one of the three outer strips added around $T_{d-3}$, and therefore is farther from the corresponding opposite corner than some vertex of $T_{d-3}$, so such vertices cannot be central.

Let $u$ be a central vertex of $T_{d-3}$ with $ecc'(u)=r_{d-3}$.
Then
$
ecc(u)=r_{d-3}+2,
$
and, by the preceding argument, no vertex outside the internal copy of $T_{d-3}$ has eccentricity less than $r_{d-3}+2$. Therefore, $u$ remains central.
Hence
$
r_d=r_{d-3}+2,
$
completing the proof.
\end{proof}

\input{sections/4_radius/radius_coro}

%% file: sections/4_radius/radius_coro.tex
As an immediate consequence of Theorem~\ref{teo:radio}, if $u$ is a central vertex of $T_d$, then
$
ecc(u)=r_d=\left\lceil \frac{2d}{3} \right\rceil,
$
and hence, for every $k\ge r_d$, the singleton set $\{u\}$ is a distance-$k$ dominating set of $T_d$. Therefore, we have the following corollary.

\begin{corollary}\label{coro:radio}
Let $d \ge 0$. Then $\gamma_k(T_d)=1$ for every
$k \ge \left\lceil \frac{2d}{3} \right\rceil.
$
\end{corollary}

Corollary \ref{coro:radio} was already proven in~\cite[Lemma~3.2]{harris_2020}, where it is expressed as
$
d \le \left\lfloor \frac{3k}{2} \right\rfloor
$
in the context of $(t,1)$-broadcast domination, with $t=k+1$. However, their proof is based on a geometric argument describing the largest triangular subgraph that fits inside a broadcast hexagon of radius $t-1$. In contrast, our formulation follows immediately from the radius of $T_d$.